\documentclass[reqno]{amsart}
\usepackage{amssymb,amscd}
\usepackage[all]{xy}
\numberwithin{equation}{subsection}

\swapnumbers
\theoremstyle{plain}
\newtheorem{thm}[subsection]{Theorem}
\newtheorem{prop}[subsection]{Proposition}
\newtheorem{lemma}[subsection]{Lemma}

\theoremstyle{definition}

\theoremstyle{remark}

\numberwithin{subsection}{section}
\numberwithin{equation}{section}

\newcommand{\Z}{\mathbb{Z}}
\newcommand{\Q}{\mathbb{Q}}
\newcommand{\C}{\mathbb{C}}
\renewcommand{\P}{\mathbb{P}}
\newcommand{\into}{\hookrightarrow}
\newcommand{\tensor}{\otimes}
\DeclareMathOperator{\gal}{Gal}
\DeclareMathOperator{\en}{End}

\begin{document}

\title{Low-dimensional factors of superelliptic Jacobians}

\author{Thomas Occhipinti}
\address{Department of Mathematics, Carleton College, 
One North College Street,
Northfield, MN 55057 USA}
\email{tocchipinti@carleton.edu}

\author{Douglas Ulmer}
\address{School of Mathematics, Georgia Institute of Technology, 
686 Cherry Street, Atlanta, GA 30332 USA}
\email{ulmer@math.gatech.edu}



\begin{abstract}
  Given a polynomial $f\in\C[x]$, we consider the family of
  superelliptic curves $y^d=f(x)$ and their Jacobians $J_d$ for
  varying integers $d$.  We show that for any integer $g$ the number
  of abelian varieties up to isogeny of dimension $\le g$ which appear
  in any $J_d$ is finite and their multiplicities are bounded.
\end{abstract}


\maketitle

\section{Introduction}

Given a family of abelian varieties, it is often interesting to ask
about the existence and multiplicity of small-dimensional abelian
subvarieties.  For example, in \cite{EkedahlSerre93}, Ekedahl and
Serre ask several questions about curves over $\C$ whose Jacobian is
isogenous to a power of an elliptic curve.

Related questions about Fermat Jacobians are discussed in
\cite{Ulmer07c}.  Let $F_d$ denote the Fermat curve of degree $d$ over
a field of characteristic zero.  In \cite[Thm.~3.2]{Ulmer07c} it is
proven that for every $g>0$, only finitely many abelian varieties up
to isogeny of dimension $\le g$ appear as subvarieties of the Jacobian
of $F_d$ for any $d$.  The goal of this note is to generalize that
result to a large class of superelliptic curves.  Our results apply
over any field of characteristic zero, but for simplicity of language
we work over the complex numbers $\C$.

Fix a polynomial $f\in\C[x]$ of degree $n>0$ and let $e$ be the largest
positive integer such that $f=f_0^e$ where $f_0$ is a polynomial.  If
$d$ is a positive integer which is relatively prime to $e$, then the
polynomial $y^d-f(x)$ is irreducible.  We may thus define $C_{f,d}$
to be the smooth, projective curve over $\C$ associated to the affine
plane curve
\begin{equation}\label{eq:curve}
y^d-f(x)=0.
\end{equation}
If $f_0$ is linear, then the genus of $C_{f,d}$ is zero for all $d$
prime to $e$.  On the other hand, if $\deg f_0>1$, then the genus of
$C_{f,d}$ tends to infinity with $d$.

Here is our main result:

\begin{thm}\label{thm:main0}
  For a fixed polynomial $f$ as above and a fixed positive integer
  $g$, only finitely many abelian varieties up to isogeny of dimension
  $\le g$ appear as subvarieties of the Jacobians of $C_{f,d}$ as $d$
  varies over all positive integers relatively prime to $e$.
\end{thm}

This result can be used to give more examples of isotrivial elliptic
curves (and abelian varieties) over $\C(t)$ which have bounded
Mordell-Weil ranks in the tower of extensions $\C(t^{1/d})$, along the
lines of the discussion in \cite[\S4]{Ulmer07c}.  We leave the details
as an exercise for the reader.

We fix $f$ and write $J=J_d$ for the Jacobian of $C_{f,d}$.  It will
be convenient to reformulate the theorem in terms of the ``new part''
of $J$.  Note that if $d'$ divides $d$, then there is a surjective
morphism $C_{f,d}\to C_{f,d'}$ which induces a homomorphism of
Jacobians $J_{d'}\to J_d$.  We define $J^{new}=J^{new}_d$ to be the
quotient of $J_d$ by the sum of the images of these morphisms for all
proper divisors $d'$ of $d$.  It is not hard to show that $J_d$ is
isogenous to the product of $J^{new}_{d'}$ with $d'$ ranging over all
divisors of $d$.  Theorem~\ref{thm:main0} is thus implied by the
following result.

\begin{thm}\label{thm:main}
  For all integers $n>0$ and $g>0$, there is a constant $d_0$
  depending only on $n$ and $g$ with the following property: If $f$ is
  a polynomial of degree $n$ and $e$ is the largest positive integer
  such that $f=f_0^e$, and if $d$ is relatively prime to $e$ and
  $>d_0$, then $J^{new}_d$ has no abelian subvarieties of dimension
  $\le g$.
\end{thm}

Note that Theorem~\ref{thm:main} is stronger than
Theorem~\ref{thm:main0} because it shows that the dependence on $f$ is
only through its degree.

The proof of Theorem~\ref{thm:main} will be given in
Sections~\ref{s:complex} through \ref{s:conclusion}.  At the end of
Section~\ref{s:conclusion} we discuss what we know or expect in other
contexts, such as when $f$ is a rational function or when the ground
field has positive characteristic.

We thank the anonymous referee for a careful reading of the paper and
for several useful suggestions.

\section{The complex representation}\label{s:complex}
We keep the notations and definitions of the introduction.  The curve
$C$ is a Galois cover of $\P^1$ with Galois group the $d$-th roots of
unity $\mu_d$.  In terms of the equation \eqref{eq:curve}, $\mu_d$ acts
via multiplication on the coordinate $y$.  We will study the action of
$\zeta$ on the regular 1-forms of $C$, or equivalently on the regular
1-forms of $J$.  This is sometimes called the complex representation
of $J$.

Fix once and for all a primitive $d$-th root of unity $\zeta$.
We have a direct sum decomposition
$$H^0(J,\Omega^1)=H^0(C,\Omega^1)=\bigoplus_{j=0}^{d-1}V_j$$
where $V_j$ is the subspace where $\zeta$ acts by multiplication by
$\zeta^j$.  It will be convenient to view the subscripts $j$ as lying
in $\Z/d\Z$.

Note that since $C_{f,d'}$ is the quotient of $C_{f,d}$ by the
subgroup of $\mu_d$ generated by $\zeta^{d'}$, we may identify
$H^0(C_{f,d'},\Omega^1)$ with the subspace 
$$\bigoplus_{j=0}^{d'-1}V_{(d/d')j}.$$
It follows that we have 
$$H^0(J^{new},\Omega^1)\cong\bigoplus_{j\in(\Z/d\Z)^\times}V_j.$$
In particular, the roots of the characteristic polynomial of $\zeta$
on $H^0(J^{new},\Omega^1)$ are primitive $d$-th roots of unity.

A crucial point in our analysis is that the dimensions of the $V_j$
are not evenly distributed.  More precisely:

\begin{prop}\label{prop:ill-distribution}
  If $j\in(\Z/d\Z)^\times$ has least positive residue
  $<d/n=d/\deg(f)$, then $V_j=0$.
\end{prop}

\begin{proof}
  A standard calculation gives a basis of 1-forms on $C$ consisting of
  forms of the shape $g(x)dx/y^k$, and the dimension of $V_j$ is the
  number of basis elements with $k\equiv -j\pmod d$.  See
  \cite[\S6]{Archinard03} for a detailed account.  The proposition is
  an immediate consequence of the dimension formula
  \cite[Thm.~6.7]{Archinard03}.

  For the convenience of the reader, we restate
  \cite[Thm.~6.7]{Archinard03} using our notation.  For this, write
$$f_0=c\prod_{i=1}^m(x-\lambda_i)^{e_i}.$$
Then the formula in question says that if $\gcd(j,d)=1$ and $d$ does
not divide any of $ee_1,\dots,ee_m$ nor $n$, then
$$\dim V_j=-\left\langle\frac{jn}d\right\rangle
+\sum_{i=1}^m\left\langle\frac{jee_i}d\right\rangle.$$ Here $\langle
r\rangle$ denotes the fractional part of a real number $r$.  It
follows that if $d>n$ and $j<d/(e\max\{e_i\})\le d/n$, then $\dim
V_j=0$, and this is exactly the assertion of the proposition.
\end{proof}

\section{Consequences of a small factor}
In this section we show that if $J^{new}$ has an abelian subvariety of
small dimension, then the characteristic polynomial of $\zeta$ on
$H^0(J^{new},\Omega^1)$ has a factor with coefficients in a number
field of small degree.  More precisely:

\begin{prop}\label{prop:small-factor}
  Suppose that $J^{new}$ has an abelian subvariety of positive
  dimension $\le g$.  Then there exists a number field $K\subset\C$ of
  degree $\le2g$ and a polynomial $P$ of positive degree with
  coefficients in $K$ such that $P$ divides the characteristic
  polynomial on $\zeta$ on $H^0(J^{new},\Omega^1)$.
\end{prop}

Note that we may take the polynomial $P$ to be irreducible over $K$.
Note also that the roots of $P$ are primitive $d$-th roots of unity.

\begin{proof}
  Suppose that $A\subset J^{new}$ is an abelian subvariety of positive 
  dimension $\le g$.  We may assume that $A$ is simple.  Let $a>0$ be
  the integer such that $J^{new}$ is isogenous to $A^a\oplus B$ where
  $B$ is an abelian variety with no factors isogenous to $A$.  Fixing
  an isogeny yields an identification
$$H^0(J^{new},\Omega^1)\cong H^0(A^a,\Omega^1)\oplus
H^0(B,\Omega^1).$$ 
Since $A$ and $B$ have no isogeny factors in common, the action of
$\zeta$ preserves the subspace $H^0(A^a,\Omega^1)$.  Our claim will
follow from an inspection of the characteristic polynomial of $\zeta$
on $H^0(A^a,\Omega^1)$.

Let $D$ be the rational endomorphism algebra of $A$:
$D=\en(A)\tensor\Q$.  Then it is well known \cite[\S19]{MumfordAV} that
$D$ is a simple algebra whose center $K$ is a number field of degree
$\le2g$.  Let $[D:K]=f^2$.  The endomorphism algebra of $A^a$ is
$M_a(D)$, and the irreducible complex representations of this algebra
are indexed by the embeddings of $K$ into $\C$.  Given an
embedding, the corresponding representation is the composition
$$M_a(D)\into M_a(D)\tensor_K\C\cong M_{af}(\C)$$
(where the embedding is used to form the tensor product) acting by
left multiplication on $\C^{af}$.  If $\delta\in M_a(D)$, the eigenvalues
of the action of left multiplication by $\delta$ on the $K$-vector space
$M_a(D)$ are the eigenvalues of the corresponding element of
$M_{af}(\C)$ each repeated $af$ times.  It follows that the
characteristic polynomial of the element $\delta$ in either representation
has coefficients in $K$.  Therefore, we may take $P$ to be the
characteristic polynomial of $\zeta$ on an irreducible constituent of
$H^0(A^a,\Omega^1)$ viewed as a representation of $M_a(D)$.
\end{proof}

\section{Galois theory}
Next we consider the distribution of roots of a polynomial with
coefficients in a small field and with a root which is a root of
unity. More precisely:

\begin{lemma}\label{lemma:Galois}
  Let $P$ be an irreducible polynomial with coefficients in a number
  field $K$ of degree $\le2g$ over $\Q$.  Suppose that $P$ has
  primitive $d$-th roots of unity as roots.  Then there exists a
  subgroup $H\subset(\Z/d\Z)^\times$ of index $\le 2g$ such that the
  roots of $P$ are of the form $\zeta^{ab}$ for a fixed
  $a\in(\Z/d\Z)^\times$ and $b\in H$.
\end{lemma}

\begin{proof}
  This is a simple exercise in Galois theory:  Since the roots of $P$
  lie in $\Q(\mu_d)$, we may assume that $K$ is a subfield of
  $\Q(\mu_d)$.  The splitting field of $P$ is $\Q(\mu_d)$, so the
  roots of $P$ are some fixed primitive $d$-th root $\zeta^a$ and its
  conjugates under $H=\gal(\Q(\mu_d)/K)$.
\end{proof}

\section{Equidistribution}
The last step in the proof of the main theorem is to note that
subgroups of $(\Z/d\Z)^\times$ with small index are in a
certain sense equidistributed.  More precisely:

\begin{prop}\label{prop:equidistribution}
  Given positive integers $n$ and $g$, there exists an integer $d_0$
  depending only on $n$ and $g$ such that if $d>d_0$ and
  $H\subset(\Z/d\Z)^\times$ is a subgroup of index $2g$, then every
  coset of $H$ has an element whose least residue lies in the interval
  $(0,d/n)$.
\end{prop}

\begin{proof}
  The proposition follows from the following more precise statement:
  If for each $d$ we have a subgroup $H_d\subset G_d:=(\Z/d\Z)^\times$
  of index $\le 2g$, then as $d\to\infty$ the subsets $\{e^{2\pi i
    b/d}|b\in H_d\}$ become equidistributed with respect to Haar
  measure on the circle.

  This equidistribution statement follows easily from the Weyl
  equidistribution criterion (see \cite{Weyl16} or
  \cite[4.2]{SteinShakarchiFA})  and standard estimates for Gauss
  sums (see \cite[2.1.6]{CohenNT1}).
  What is to be shown is that for all integers $a>0$, the quantity 
$$\frac{1}{|H_d|}\sum_{b\in H_d}\zeta_d^{ab}$$
tends to zero as $d\to\infty$.  We have
\begin{align*}
\left|\frac{1}{|H_d|}\sum_{b\in H_d}\zeta_d^{ab}\right|
&=\frac{1}{|H_d|}\left|\sum_{b\in G_d}
   \left(\frac{1}{[G_d:H_d]}\sum_{\chi\in\widehat{G_d/H_d}}\chi(b)\right)
    \zeta_d^{ab}\right|\\
&\le\frac{1}{\phi(d)}\sum_{\chi\in\widehat{G_d/H_d}}
   \left|\left(\sum_{b\in G_d}\chi(b)\zeta_d^{ab}\right)\right|\\
&\le\frac{2g}{\phi(d)}\sqrt{ad} 
\end{align*}
where the last step uses standard estimates for Gauss sums.  The
last displayed quantity goes to zero as $d\to\infty$.  This
establishes the desired equidistribution.  It follows that for large
$d$, each coset of $H_d$ has an element whose least residue lies in
$(0,d/n)$, as desired.
\end{proof}

\section{Conclusion and complements}\label{s:conclusion}

\subsection{Proof of Theorem~\ref{thm:main}}
Fix integers $n$ and $g$ and let $d_0$ be the integer associated to
$n$ and $g$ by Proposition~\ref{prop:equidistribution}.  If necessary,
we may increase $d_0$ so that $d_0>n$.  Fix a polynomial $f$ of degree
$n$, form the curves $C_{f,d}$ for varying $d$, and consider the new
parts of their Jacobians $J^{new}_d$.

Suppose that $A$ is an abelian variety of dimension $\le g$ which
appears in $J^{new}_d$ for some $d>d_0$.  By
Proposition~\ref{prop:small-factor}, the characteristic polynomial of
$\zeta$ on $H^0(J^{new}_d,\Omega^1)$ has a factor $P$ with
coefficients in a number field $K$ of degree $\le2g$ over $\Q$.  By
Lemma~\ref{lemma:Galois} and Proposition~\ref{prop:equidistribution},
$P$ has a root of the form $\zeta^b$ with $0<b<d/n$ and $\gcd(b,d)=1$.
This implies that $V_b\neq0$, in contradiction to
Proposition~\ref{prop:ill-distribution}.  Therefore, no such $A$
exists, and this proves Theorem~\ref{thm:main}.\qed

\subsection{Effectivity}
It would be interesting to make the bound $d_0$ effective.  For $g=1$
and $n=2$, previous unpublished work of the first author shows that
Theorem~\ref{thm:main} holds with $d_0=10^{11}$. It can then be shown
computationally that in fact $d_0=24$ is sufficient.  At least insofar
as Proposition~\ref{prop:equidistribution} is concerned, this cannot
be improved further because $\{1,5,7,11\}$ is an index $2$ subgroup of
$(\Z/24\Z)^\times$.

We also note that if $f$ has degree $\ge3$ and no repeated roots, and
if $d$ is relatively prime to $n$ and sufficiently large ($d>24$
suffices), then by \cite[Cor.~1.11]{XueZarhin10}, $J_d$ is not
isogenous to a product of elliptic curves.

\subsection{Characteristic $p$}
We note that the restriction to characteristic zero is essential:
Theorem~\ref{thm:main0} and similar results are spectacularly false in
characteristic $p$.  Indeed, Tate and Shafarevich
\cite{TateShafarevich67} showed that the Jacobian of the Fermat curve
of degree $p^f+1$ is isogenous to a product of elliptic curves.  More
generally, a construction in \cite[\S10]{Ulmer07b} yields many
superelliptic curves whose Jacobians have an elliptic curve as factor
with large multiplicity in the new part of their Jacobians.

All these examples involve supersingular elliptic curves.  Comparing
with \cite[Thm.~3.3]{Ulmer07c}, one might hope that introducing
restrictions like ordinarity or positive $p$-rank on small dimensional
factors $A$ might yield a true variant of Theorem~\ref{thm:main0}.

\subsection{Rational functions}
Finally, it would be natural to allow $f$ to be a rational function,
not just a polynomial, in Theorem~\ref{thm:main0}.  We expect the
result to continue to hold in this generality, but our proof breaks
down because allowing poles in $f$ makes the ``ill-distribution''
result of Proposition~\ref{prop:ill-distribution} false.  For example,
if $f$ is a rational function with distinct zeroes and poles, then the
dimensions of all of the $V_j$ with $j\in(\Z/d\Z)^\times$ are the
same.  Of course, if $f$ is a rational function, the curve defined by
$y^d=f(x)$ can also be defined (using a suitable change of variables)
by an equation $y^d=h(x)$ with $h$ a polynomial, but now $h$ will have
large degree and this also spoils our argument.

Results in the opposite direction, namely examples where a given
small-dimensional abelian variety appears in $J^{new}_d$ for
infinitely many $d$, would also very interesting due to their
relevance for ranks of abelian varieties and questions like those of
Ekedahl and Serre.

\bibliography{database}{}
\bibliographystyle{alpha}

\end{document}